\documentclass[12pt]{article}

\usepackage[english]{babel}

\usepackage[a4paper,top=2cm,bottom=2cm,left=3cm,right=3cm,marginparwidth=1.75cm]{geometry}

\usepackage{amsmath}
\usepackage[colorlinks=true, allcolors=blue]{hyperref}
\usepackage{graphicx}
\usepackage{amssymb}
\usepackage{amsthm}
\usepackage{subcaption}

\theoremstyle{definition}
\newtheorem*{definition}{Definition}

\newcommand{\es}{\text{.}}
\newcommand\N{{\mathbb N}}

\newcommand\OB{{\mathbb O}}
\newcommand\SB{{\mathbb S}}
\newcommand\PB[1]{{\mathbb P}_{#1}}

\newcommand{\e}[1]{{\rm e_{#1}}} 
 
\newcommand{\Cl}[1]{{\rm Cl}({#1})}

\newcommand{\Spin}[1]{{\rm Spin}({#1})}
\newcommand{\Sharp}[1]{{\rm Sharp}({#1})}

\newcommand{\PG}[2]{{\rm PG}(#1,#2)}
\newcommand\degs{^\circ}
\newcommand\tab[1]{\begin{table}[#1]\centering\vskip -2ex}
\newcommand\tac[3]{\begin{table}[#1]\caption{\label{tab:#2}#3}\centering\vskip -2ex}
\newcommand\tae{\end{table}}
\newcommand\tar[1]{Table~\ref{tab:#1}}

\newcommand\eqb[1]{\begin{equation}\label{eqn:#1}}
\newcommand\eqe{\end{equation}}
\newcommand\eqr[1]{{\rm (\ref{eqn:#1})}}
\newcommand\eab[1]{\[\begin{array}{#1}}
\newcommand\eae{\end{array}\]}
\newcommand\esb{\[\begin{split}}
\newcommand\ese{\end{split}\]}
\newcommand\elb[1]{\[\label{eqn:#1}\begin{aligned}}
\newcommand\ele{\end{aligned}\]}
\newcommand\fib[3]{\begin{figure}#1\begin{center}\begin{tabular}{c}
    \includegraphics#2{figures/#3}}
\newcommand\fie[2]{\end{tabular}\end{center}\vskip -4ex
    \caption{\label{fig:#1}#2}\end{figure}}
\newcommand\fig[1]{Figure~\ref{fig:#1}}

\newtheorem*{theorem}{Theorem}

\title{Clifford Algebra Calibration Post-Quantum Cryptography}
\author{G.~P.~Wilmot, J.~Chappell, D.~K.~Abbott}
\date{}

\begin{document}
\maketitle
\begin{abstract}
   The double lock scenario allows a secret to be transferred from one person to another without having previously exchanged secrets such as private/public keys. This can be realised by rotations of generalised calibrations in Clifford algebra. Calibrations map to Cayley-Dickson algebras and generalising this provides the mathematical structure enabling the enormous code space needed for a cryptosystem. Commuting rotations in large dimensional spaces are used but the complexity comes from the message being embedded in generalised calibrations. This system uses ideals of Clifford algebras and quasi-algebras that are sometimes embedded in Cayley-Dickson algebras making it a candidate for post-quantum cryptography. Understanding whether the new mathematics of generalised calibrations can produce a secure cryptosystem is the challenge presented in this manuscript.
   \par\vspace{1ex}\noindent{\bf Clifford algebras, quasi-Cayley-Dickson algebras, calibrations, post-quantum cryptography}
\end{abstract}
\section{Introduction}
The double lock box is an example of symmetric encryption whereby a secret is passed between two agents without prior knowledge being exchanged. It was theorised by Kish-Sethuraman~\cite{Kish} (KS) but was already superseded by the Diffie-Hellman system, which requires fewer transmissions to establish a secure channel. In fact, the factorisation of large primes used in the original implementation of Diffie-Hellman can be used in the KS system~\cite{Klappenecker} with less efficiency. But the double lock box admits a solution that is akin to information-set decoding, so can be considered for post-quantum cryptography. 
\begin{figure}[ht]\begin{center}\begin{tabular}{c}
    \includegraphics[width=8cm]{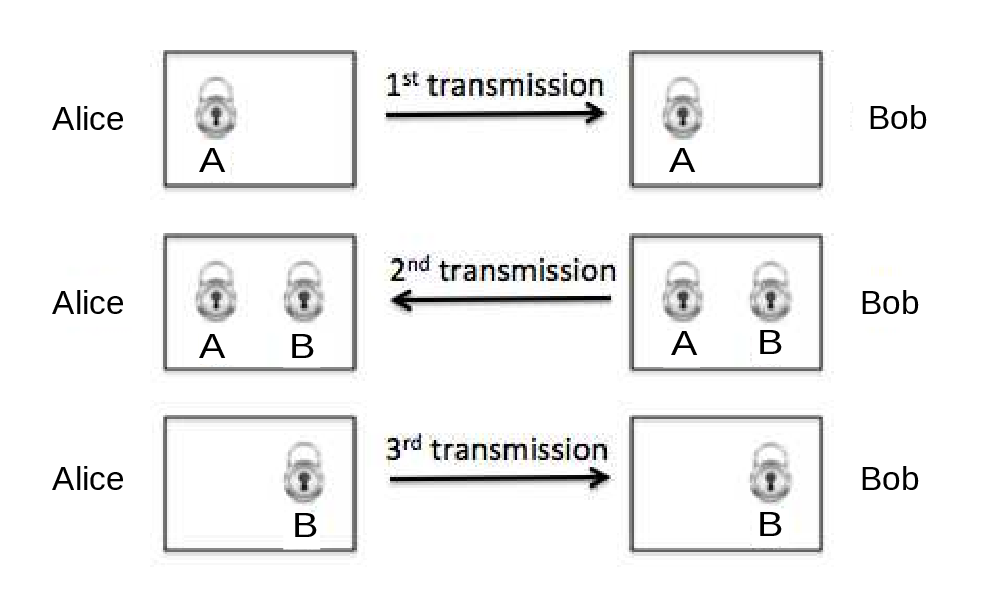}
\fie{pad}{Transmission scheme}
The KS cipher, as described by Chappell~\cite{Chappell}, and shown in \fig{pad}, involves Alice constructing a secure, tamper evidence box with double latches. A secret is placed inside and a padlock on the outside. This is posted to Bob who adds another padlock in parallel to the first padlock and returns the box. Alice then unlocks the first lock and again forwards the box to Bob who unlocks the second lock to retrieve the secret. The digital version of this scheme would transmit the digital box in the open, via email for example. To enable such a scheme, KS show that the lock operators, ${\mathcal A}$ and ${\mathcal B}$, must commute. If the message in the box is represented as $m$, then
  \eqb{AB} {\mathcal A}({\mathcal B}m) = {\mathcal B}({\mathcal A}m)\es \eqe
The construct used here is orthogonal rotations so that ${\mathcal A}{\mathcal B}={\mathcal B}{\mathcal A}$, implying no intersection of orthogonal dimensions between ${\mathcal A}$ and ${\mathcal B}$. 

The orthogonal rotations for $\Cl{n}$ use the Spin group with basis $\e{ij}$ such that $\phi'=\theta_{ij}(\phi)={\rm R}_{ij}\phi{\rm R}^{-1}_{ij}$, where ${\rm R}_{ij} = \cos(\frac\pi4) +\e{ij}\sin(\frac\pi4)$, $i\ne j\in\N_1^n$. The rotated calibration $\phi'$ has the form of $\phi$ but with different signs and 3-forms but the inverse rotation, $\theta_{ji}$ will return $\phi$. Now since orthogonal rotations commute $[\e{ij},\e{kl}]=0$, for $k,l\ne i,j$, then
\[ \theta_{lk}(\theta_{ji}(\theta_{kl}(\theta_{ij}(\phi)))) = \phi\es \]
This is a symmetric encryption scheme where Alice has used rotation $\theta_{ij}$ and Bob has used $\theta_{kl}$ without prior knowledge other then agreement of orthogonality. The size can be inferred from the message size. For the very small example of seven bits there are only 3840 possible outcomes and less than 21 total rotations, which are the keys to the locks in this encryption scheme.

Equation \eqr{AB} is usually interpreted with $m$ being a large integer or a vector so that ${\mathcal A}m$ has the same form or properties. But Clifford algebra offers another message format that has properties whereby ${\mathcal A}m$ is not a linear combination of $m$. These are generalised calibrations that are akin to using information-set codes for each bit of the message. The rotations not only change the codes but also the values of the bits. The number of codes increases as $O(n)$ over the number of rotations, $\binom{n}2$, which has $O(n^2)$ complexity increase.

Calibrations in the positive definite Clifford algebra, $\Cl{n}$, $n=2^N-1$, $N>1$ are shown to map to all Cayley-Dickson algebras~\cite{Wilmot3}. Calibrations specify faces and parity of generalised polygons or polytops with $n$ completely connected  vertices by arrows called a directed $(n-1)$-simplex. Calibrations represent coverage of all edges of the simplex by faces that only intersect at vertices, which is designated independent faces because their parities are independent. This means any pair of dimensions are uniquely associated with another dimension, not unique for $n>3$, which defines a cross product. This, in turn, defines a multiplication table and hence an algebra. Some of these generalised calibrations of $\Cl{n}$ map to Cayley-Dickson algebras but others are shown to map to quasi-Cayley-Dickson algebras~\cite{Wilmot1, Wilmot3}. This mapping is only important in the context of the cryptosystem in order to prove that the generalised calibrations are well defined. 
\begin{definition} The generalised calibrations are divided into calibrations that map to Cayley-Dickson algebras and quasi-calibrations that map to quasi-Cayley-Dickson algebras. Reference~\cite{Wilmot3} showed that for $N=3$, calibrations transform to other calibrations under any rotation and hence, since rotations are invertible, quasi-calibrations transform amongst themselves. Calibrations are a much smaller subset of generalised calibrations and this needs to be taken into consideration later.
\end{definition}
\begin{definition}
A quasi-Cayley-Dickson algebra is identified by the $\PB{k}$ code where $k$ is the number of non-associative unique triads, $[a,b,c]\ne0$, for $a<b<c$, and the product $[a,b,c]=(ab)c-a(bc)$ is defined by the terms of the calibration $\pm\e{ijk}$ s.t. if $a=\e{i}$ and $b=\e{j}$ then $ab=\pm\e{k}$ and $\e{j}\e{i}=\mp\e{k}$. The number of quasi-Cayley-Dickson algebras is six for $\Cl7$ and $103$ for $\Cl{15}$ and is unknown thereafter but the number of primary representations of calibrations is called the {\it representation space}, which is an extremely large space that enables the security of the cryptosystem.
\end{definition}

\tac{ht}{dims}{Generalised Calibration Dimensions}
\begin{tabular}{|c|c|c|c|c|c|}\hline
Level &Dimension &Rotations &Msg Bit Length &Rep. Space\\\hline
 3 &7 &21 &7 &30 \\
 4 &15 &105 &35 &$64.8\times10^6$ \\
 5 &31 &465 &155 &$8.2\times10^{26}$ \\
 6 &63 &1953 &651 &$9.8\times10^{76}$ \\
 7 &127 &8001 &2667 &$1.8\times10^{199}$ \\
 8 &255 &32,385 &10,795 &$6.1\times10^{485}$ \\
 9 &511 &130,305 &43,435 &$1.9\times10^{1,140}$ \\
 10 &1,023 &522,753 &174,251 &$5.3\times10^{2,607}$ \\\hline
\end{tabular}\tae

The dimensions, number of rotations, sizes of the generalised calibration and representation space sizes are provided in \tar{dims} for levels up to $N=10$. The message length is the number of codes or terms in the generalised calibration. The number of codes available as 3-forms, the $\PB{k}$ code and other parameters such as the number of calibrations and quasi-calibrations are provided in \tar{pars}. 

An example provides an ideal way to introduce this encryption system. For $N=3$ with seven dimensions the first generalised calibration from~\cite{Wilmot1} is
\[ \phi_1 = \e{123} +\e{145} +\e{167} +\e{246} +\e{257} +\e{347} +\e{356}\es \]
This is a 3-form with seven terms called a primary since it has all positive terms. Any set of sign variations can be selected to represent bits with $+\rightarrow1$ and $-\rightarrow0$. 

The set of $\binom72=21$ $90\degs$ basis rotations or 2-forms in $\Cl7$, applied repeatedly, was shown in~\cite{Wilmot1} to generate 30 other 3-forms that cover all the $\binom73=35$ single 3-forms in $\Cl7$. Changing the signs of certain terms provides a calibration that maps to octonions, $\OB$, and reflections generate another 15 maps to $\OB$ thus generating the 480 representations of octonions. The remaining $30\times2^7-480$ sign variations generate six other algebras that are power-associative and called quasi-octonions and these have many representations. This is the mathematical underpinning of the calibration cryptographic technique. The 3-form terms act as information-set codes with $30\times112=2240$ possible outcomes if a quasi-calibration was the starting point. The rotations can be performed in any order and will only commute if orthogonal. Octonions are the last of the normed Cayley-Dickson algebras with higher levels being power-associative containing zero-divisors. The sedenions, $\SB$, at level $N=4$, have 15 vertices of the 14-simplex represented by the 15 basis elements of $\Cl{15}$. The possible outcomes here are $32528\times64.8\times10^6$.

The octonion example here is minimal in that only seven bits can be encoded. But \tar{dims} demonstrates the message length and number of rotations increase quickly. Given that it is the permutations of the rotations with up to $\lfloor{\binom{n}2/2}\pm1\rfloor!$ possibilities for Alice and Bob then the complexity has increased exponentially as seen by the representation space column of \tar{dims}. This column is the number of primary calibrations available, one of which has the signs in the correct order to be the cipher.

\tac{ht}{pars}{Generalised Calibration Parameters}
\begin{tabular}{|c|c|c|c|c|c|c|}\hline
Level &Codes &$\PB{k}$ Code &Internal Cal. Count &All-Even \\\hline
 3 &35 &28 &127 & 1\\
 4 &455 &252 &32768 &7 \\
 5 &4495 &2156 &$2.1\times10^{9}$ &35 \\
 6 &39,711 &18,732 &$9.2\times10^{18}$ &155 \\
 7 &333,375 &159,852 &$1.7\times10^{38}$ &651 \\
 8 &2,731,135 &1,331,820 &5.8$\times10^{76}$ &2667 \\
 9 &22,108,415 &10,902,892 &$6.7\times10^{153}$ &10,795 \\
 10 &177,910,271 &88,310,124 &$9.0\times10^{307}$ &43,435 \\\hline
\end{tabular}\tae

The columns in \tar{pars} are the number of Codes, which is the number of distinct 3-forms used in the Message Length of \tar{dims}, the $\PB{k}$ Code for the calibration and the number of generalised calibrations within a primary, which is the number of permutations of signs in each primary calibration, $2^m$ where $m$ is the message bit length. This totals to the number of calibration and quasi-calibrations. Note that $\PB{28}$ is used to represent the calibration that maps to octonions, even though octonions are not power associative. All other $\PB{k}$ codes map to power associative algebras. As stated above, there are six quasi-octonions and we later find 103 quasi-sedenions alongside sedenions with code $\PB{252}$. It is easy to see the Calibration Count is much smaller than the number of quasi-calibrations, which is analysed further because the octonion calibrations are invariant under all Pin transformations, which limits their applicability as a cryptosystem.

The mathematics that underpins this cryptosystem is derived from ideals of $\Cl{n}$ subalgebras, called $\Sharp{n}$ algebras. These algebras are new and specify the structure of Cayley-Dickson algebras by identifying non-associative rings related to the associative generalised calibrations. This is described in the following section. Section~3 describes the encryption/decryption code presented in Appendix~A. Section~4 provides mathematical tests of the code by verifying some of the numbers in \tar{pars}. This code is presented in Appendix~B and includes finding the 103 quasi-sedenion calibrations provided in Appendix~C.

\section{Mathematical Theory}
Clifford algebras are represented as directed simplices and can be visualised as projections for sizes up to $N\leq4$. For the dimensions selected here $n=2^N-1$, $N\ge2$, these is the same infinite series as the finite geometry group $\PG{N}2$\cite{Wilmot3}. For example, the projection of the 6-simplex from a point at infinity to a plane gives the Fano plane, $\PG22$, shown in \fig{fano}.
\begin{figure}[ht]\begin{center}\begin{tabular}{c}
    \includegraphics[width=5cm]{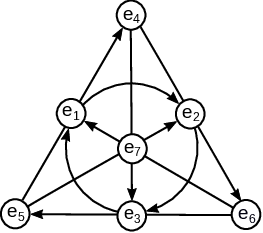}
\fie{fano}{Fano plane calibration diagram}

For the selected dimensions of $\Cl{n}$, $n=2^N-1$, $N>1$, then the number of basis rotations is $\binom{n}2$, as shown in \tar{dims}. This is also the number of edges in the $(n-1)$-simplex and is always divisible by three since $\binom{n}2$ is a product of two consecutive numbers less than a power of two and this gives the number of faces that cover each edge once. This property is the main characteristic that defines calibration structures. The number of faces of this simplex is $\binom{n}3$ and a third of these can be selected that are independent and cover all edges so the primary calibration consists $\frac13\binom{n}2$ 3-forms and this is provided in the Message Bit Length column of \tar{dims}.

The $\PB{k}$ column where $k$ is the code is derived from the number of subalgebras found in~\cite{Wilmot3} and copied in \tar{usa}. The value of $k$ for each level is the sum of each subalgebra's count multiplied by its $k$ value where $k=28$ for $\OB$.

\tac{ht}{usa}{Subalgebra Counts}
\begin{tabular}{|c|c|c|c|c|c|c|c|c|c|}  \hline   
{\bf Level}& $3$& $4$& $5$& $6$& $7$& $8$& $9$ &$10$ \\\hline
$\OB$      &1 &8  &50  &310  &2046 &14,478  &107,950  &831,470 \\
$\PB4$     &0 &7  &63  &413  &2583 &16,905  &118,251  &873,901 \\
$\PB{12}$  &0 &0  &42  &504  &4158 &30,996  &229,194  &1,729,728 \\
$\PB{14}$  &0 &0  &0   &168  &3024 &34,776  &332,640  &2,912,616 \\\hline
\end{tabular}\tae

The generalised calibration count within each primary is just the permutations of all the signs, which is $2^n-1$ where $n=2^N-1$. These are called internal generalised calibrations and are shown in \tar{pars}. This number multiplied by the representation space size from \tar{dims} provides the number of possible outputs for the encoded message. The size of the calibration representation space is given by the following.
\begin{theorem}
For $P(n)$ being the number of primary calibration representations at level $n$ then the recursive formula is
\eqb{pn} P(n) = \frac{(n -1)!}{(\frac12(n+1))!}P(n-1)\text{, \quad for $n=2^N-1$, $N\in\N_2^{\infty}$.} \eqe
\end{theorem}
\begin{proof}
Reference~\cite{Wilmot3} showed a relationship between the sedenion associative calibration and the non-associative calibration. We assume here that this correspondence occurs for all higher levels, $N>4$, which allows the later to define the automorphisms of the former. Since both calibrations have the same representation space of 64.8 million generalised calibrations, as proved by the test code in Appendix~B, then this seems highly likely. An associative calibration has $\binom{n}2/3$ 3-form terms that represent Cayley-Dickson algebras consisting of this many quaternions. The non-associative calibration consisting of $n$ terms is invertible so has a dual by multiplying it with the $\Cl{n}$ pseudoscalar to bring it into the Sharp algebra as an $((n+1)/2)$-form. Since this has even degree it generates $90\degs$ rotations that are automorphisms and define 2-form automorphisms of the Cayley-Dickson algebra. Each of the $n$ terms thus provides $((n+1)/2)!$ permutations that need to be replaced with the permutations provided by the subalgebras, which are the representations of the calibrations given by each term and have sizes $P(n-1)$. Hence the total number of permutations, $n!$, is divided by the permutations of the calibrations of the subalgebras, ${n((n+1)/2)!}/{P(n-1)}$.
\end{proof}

Generalised calibrations can be generated manually by building upon the previous level. We start with $N=2$, giving dimension $n=3$, and the calibration is $\e{123}$, which designates the $\Cl3$ quaternions, $\e{12}$, $\e{23}$ and $\e{13}$. The opposite parity, $-\e{123}$, maps to anti-quaternions where the product is $+1$. The next level is $N=3$ with dimension $n=7$ and to build upon $\e{123}$ we add 3-forms that contain pairs of indices not already used, with the lowest indices included first. Using the indices $4$, $5$, $6$ and $7$ builds the seven terms of $\phi$ above. The algorithm to generate primary calibrations is simply sum all simplex faces $\sum_{\mu\in\binom{n}3} \e{\mu_1\mu_2\mu_3}$, if any ordered pair, $(\mu_1,\mu_2), (\mu_2,\mu_3), (\mu_1,\mu_3)$, is not already included.

This generates the first primary with independent faces that cover all edges once and applying a basis $\Spin{n}$ rotation to this will generate another structure with the same property. Taking the absolute value and selecting unique structures allows this process to be repeated for all basis 2-forms until no new structures are found. Such a process is used in Section~4 to generate all 64.8 million primaries that have associative calibrations that generate $\SB$. The same code is used to generate 64.8 million 7-form non-associative calibrations, indicating that the assumption in the theorem that the two types of calibrations are related is correct.

The rotations for Alice and Bob must be orthogonal and a fixed scheme could be shared as prior knowledge or as part of a standard. Choosing small and large dimensions for rotations will favour left and right parts of the calibration so even and odd dimensions are the obvious choice. This brings the number of potential rotations to slightly less than half for all levels. For $N=4$ there are 28 all-odd rotations and 21 all-even rotations giving a ratio of $49/105=0.47$ available reflections and this ratio increases as the level increases. Even though this limits the rotations to about $24\%$ of the available rotations for each agent, these rotations can be repeated and this is needed to generate the large representation space for level $N=4$ in the test code of Appendix~B. It is also advantageous providing more opportunity for non-commuting rotations to occur.

One of the authors has pointed out that the octonion calibration contains an all-even 3-form, $\e{246}$, so initial even rotations would not encrypt this 3-form. The All-Even column of \tar{pars} shows that the number of all-even 3-forms increases as the level increases. There are no all-odd 3-forms in the first calibrations generated up to level $N=10$ so the sender of the messages, Alice, must choose all odd rotations while Bob chooses all even. Some of Alice's encrypted bits would then not be changed by Bob, making cracking the full code slightly easier. As more rotations are performed then all-odd cases increase and all-even 3-forms decrease evening out this relationship so that on average about 15\% of 3-forms will not by changed by Bob. Worse than this, after Alice un-rotates the third message, the all-even 3-forms will be in the clear or unencrypted. Both circumstances can by avoided if Alice avoids the all-even 3-forms in the original message by spreading the message or placing the bits at these location at the end of the message. Bob will need to take note of these all-even 3-forms in the first transmitted message and apply the same scheme. Hence the efficiency of the encryption scheme has decreased from $25\%$ to $21\%$ but increasing the number of repeats of the rotations will compensate for this.

The other special case mentioned in the introduction is that if the encoded bits generate a calibration then the output within the representation space is diminished. This is the case for octonions but for sedenions it is found that rotations within Fano planes and between $\OB$ planes or between $\PB4$ planes has this property. But basis $\Spin{15}$ rotations that mix $\OB$ and $\PB4$ Fano planes change the sedenion calibration into a quasi-calibration. Of the 105 basis rotations 49 have this property, covering all-even and all-odd rotations, which means the chance of starting with a calibration and not rotating to a quasi-calibration is very small. The test code mentioned below needs to rotate calibrations many times in order to find all possible 64.8 million primaries so an alternative scheme to increasing the level would be to repeat the random assignment of all possible rotations a few times to increase complexity without much extra work.

The test code uses ASCII characters ending in a control-d and the remainder of the message length is left un-randomised. Having a better termination scheme and randomising the remaining bits should be considered. Such schemes and those to avoid all-even 3-form triads are not included in the code provided. It is also imagined that ASCII is not the ideal encoding.

\section{Cryptosystem Description}
The Python code is presented in Appendix~A. To create an encrypted message, run the following command line for Windows (for Linux replace py with python or python3),
\begin{verbatim}
  py CalCrypt.py N <base filename> <cipher text>
\end{verbatim}
Thereafter, the output file is transmitted to the other agent and the saved file processed as
\begin{verbatim}
  py CalCrypt.py <saved filename>
\end{verbatim}
On the first usage, the text provided on the command line is packed into single 7-bit bytes, rotated a random number of times with odd dimension rotations and added to the calibration structure, ready for transmission. After transmission, each receiver must unpack the transmission into the $\langle$saved filename$\rangle$ for further processing.

The second usage rotates the calibration structure with random even dimension rotations. Both sets of rotations are stored and used in reverse to un-rotate the calibration structure on subsequent processing stages. On the third usage of this command, the decoded cipher is written to the output file. This sequence can be run on the same machine for investigation purposes.

The level $N=8$ with the 255 dimensions, represented by a single byte is likely to provide a practical encryption scheme. It has a code length of 10,795 bits, a transmission length of 33KB including 1 bit for each sign and stored agent rotation keys selected from about a quarter of the $\binom{255}2=$32,385 possible rotations. Starting with any calibration structure then the secret is encoded in the 10,795 signs of the 3-form structure. 

The timings for an 13th Gen Intel(R) Core(TM) i9-13900KF are 2.3 seconds to create the calibrations and 25.8 seconds to perform Alice's rotations. The later timing will be the same for all subsequence steps. The calibration time could be shortened using multiple threads.

\section{Test Code Description}
The Python code is presented in Appendix~B and has only one command line,
\begin{verbatim}
    py CalTest.py [options] N
\end{verbatim}
Use the``-'' option to view the list of options available.

The case $N=3$ will generate a list of the 30 primary associative calibrations in 7 dimensions in string format in one file and a structure specifying the six quasi-octonion codes and first calibration in a separate file. The format of this second file is:
\begin{verbatim}
    [ ([Pcode, ...], [Pindex, ...], Uindex, primary]), (...) ]
\end{verbatim}
where Pcode is the $\PB{k}$ code, Pindex is the signed offset for this Pcode within the primary, Uindex is the signed offset for the first calibration and the primary is represented as string format, which are triples of hexadecimal digits for each 3-form. The Pcodes are unique and represent the first primary calibration and the first occurrence at the beginning of this primary is provided by the respective Pindex. This takes less than a second to run.

The case $N=4$ generates the 64.8 million primary associative calibrations in 15 dimensions and the 103 quasi-sedenion quasi-calibrations. The output files are in batches of 2 million calibrations in order for Windows platforms to be able read the files. The files have the same format as above but takes a long time to run, as discussed below.

The case $N=7$ generates the 64.8 million non-associative calibrations for 15 dimensions. This again writes successive files with 2 million calibrations each with 105 hexadecimal characters per line but these are interpreted as 15 lots of 7 digits for each 7-form. There are no quasi-calibrations generated with this option. 

The two 15-D cases take many days to run on the HPC Phoenix system at Adelaide University. Since this system has a time limit of 3 days then options are available to process both the calibration generation phase and quasi-calibration phase in batches. The Processed Size reported in a previous run is used in the ``-l'' option and a time of one hour less than 3 days is used in the ``-t'' option to stop processing and write the results to the output files. Using the ``-p4'' option reads these files before processing starts. The initial run uses a Processed Size value of zero. The following slurm command is used to run the code on the Phoenix HPC.
\begin{verbatim}
    sbatch -p icelake --mem 80GB -n 70 --time 2:59:59 \
      --export=ALL,LEVEL=4,TIME=71,LAST=<Processed Size> calCrypt.sh
\end{verbatim}

The calCrypt.sh bash script has the following contents.
\begin{verbatim}
    ./CalTest.py -vc $$(($SLURM_NPROCS -1)) -t $TIME -l $LAST \
       -p $((($LAST==0) ?3 :4)) -d "/hpcfs/users/<user>" $LEVEL
\end{verbatim}
The process is repeated until the Process Size no longer increases. The processing results for the 103 quasi-calibrations at level $N=4$ are presented in Appendix~C. These results were originally generated using a Clifford algebra calculator and was able to generate all 15-D associative calibrations within 3 days using 80 cores and the 103 quasi-calibrations in another 2.5 days. The difference in time is that the calculator is optimised for 15-D while the cryptocode is unlimited up to practical time and memory resource limits. Hence this code is presented as a test of the system's ability to generate valid quasi-calibrations at known levels.

The Clifford/Cayley-Dickson algebras calculator is also written in Python and this was then translated to the standalone CalCrypt and CalTest packages. The github URL for the calculator and the code presented in Appendices~A and~B is\par\centerline{\url{https://github.com/GPWilmot/geoalg}.}

\section{Summary}
This cyptosystem has the ingredients to provide a post-quantum cryptosystem. The only prior knowledge required is the encoding of the binary string being transferred. The code presented here used ASCII for easier demonstration but any encoding scheme could be included in the standard. This paper presents this cryptosystem for the cryptography community to analyse in order to evaluate its potential.

The authors acknowledges the support of an Australian Government Research Training Program Scholarship. Support from the Australian Research Council (FL240100217) is also gratefully acknowledged.

\newpage
\section*{Appendix A - Cryptosystem Code}
{\fontsize{10pt}{10pt}\selectfont\begin{verbatim}
#!/usr/bin/env python3
#####################################################################
## File: CalCrypt.py - symmetric encryption - G.P.Wilmot (c) Mar 2026
#####################################################################
import sys, random, time

class CalCrypt():
  """Class to create quasi-calibration at level N=3..10 and encrypt
     with setCode() and rotate() or decrypt using reverse() first.
     Calibrations map to Cayley-Dickson algebras and the Sharp(N)
     algebras in Cl(n), n=2**N-1, calculate them for all N>1."""

  def __init__(self, N, harpCal=None):
    """Create the initial calibration structure for this level."""
    self.__checkType(N, int, "init", (3,10))
    n = int(pow(2, N)) -1
    if harpCal is None:   # Build calibration
      got = set([])
      cal = []
      for tri in self.__comb(n, 3, list(range(1, n +1))):
        pair1 = tuple(tri[0:2])
        pair2 = tuple(tri[1:3])
        pair3 = (tri[0], tri[2])
        if pair1 not in got and pair2 not in got and pair3 not in got:
          got.update((pair1, pair2, pair3))
          cal.append(tri)
    else:
      cal = harpCal
    self.__level = N
    self.__dim = n
    self.__msgLen = len(cal)   # (n,2) //3 = n!/(n-2)!/2 //3/def rot
    self.__cal = list(list(_) for _ in cal)
    self.__sgns = [1] *(self.__msgLen)  # Key results

  def __str__(self):
    return str({"level":   self.__level,  "dimension": self.__dim,
                "msg.len": self.__msgLen})

  def rotate(self, array):
    """Pass pairs of dimensions to rotate the code."""
    maxRot = self.__msgLen *3 //2 +1
    self.__checkType(array, (list, tuple), "rotate", length=(1, maxRot))
    list(self.__checkType(_, (list, tuple), "rotate", length=2) for _ in array)
    list(self.__checkType(_[0], int, "rotate", (1, self.__dim)) for _ in array)
    list(self.__checkType(_[1], int, "rotate", (1, self.__dim)) for _ in array)
    self.__rotate(array)
  def __rotate(self, array):
    for pair in array:
      for tri in enumerate(self.__cal):    # Rotate calibration structure
        sgn = self.__sgns[tri[0]]
        rot = tri[1][:]
        for val in enumerate(pair):
          if val[1] in tri[1]:
            idx = tri[1].index(val[1])
            rot[idx] = pair[1 -val[0]]
            if val[0]:
              sgn = -sgn                   # Change signs for 2nd rot. index
        sgn = self.__sort(rot, sgn)
        self.__cal[tri[0]] = rot
        self.__sgns[tri[0]] = sgn
    self.__sort(self.__cal, 1, self.__sgns)

  def reverse(self, array):
    """Return the array for rotate to decrypt instead of encrypt signs."""
    maxRot = self.__msgLen *3 //2 +1
    self.__checkType(array, (list, tuple), "reverse", length=(1, maxRot))
    list(self.__checkType(_, (list, tuple), "reverse", length=2) for _ in array)
    return list([_[1], _[0]] for _ in reversed(array))


  def setCode(self, signs, cal=None):
    """Set signs and optionally the calibration structure."""
    self.__checkType(signs, (list, tuple), "setCode", length=self.__msgLen)
    list(self.__checkType(_, int, "setCode", (-1,1)) for _ in signs)
    if cal is not None:
      self.__checkType(cal, (list, tuple), "setCode", length=self.__msgLen)
      list(self.__checkType(_, (list,tuple), "setCode", length=3) for _ in cal)
      n = self.__dim
      for idx in (0, 1, 2):
        list(self.__checkType(_[idx], int, "setCode", (1, n)) for _ in cal)
    self.__setCode(signs, cal)
  def __setCode(self, signs, cal):
    self.__sgns = list(signs)
    if cal is not None:
      self.__cal = list(cal)

  def getCode(self):
    """Return signs and the calibration structure."""
    return (self.__sgns, self.__cal)

  def __sort(self, array, sgn, signs=None):
    """Sort the list changing parity for pair swaps."""
    noMore = True
    while noMore:              # Bubble sort - needed to not expose rotations used
      noMore = False
      for idx in range(1, len(array)):
        if array[idx -1] > array[idx]:
          tmp = array[idx -1]; array[idx -1] = array[idx]; array[idx] = tmp
          sgn = -sgn
          if signs:
            tmp = signs[idx -1]; signs[idx -1] = signs[idx]; signs[idx] = tmp
          noMore = True
    return sgn

  def comb(self, n, r, basis):
    """Yield the combinations of r in n elements for basis of length n."""
    self.__checkType(n, int, "comb", (1,0))
    self.__checkType(r, int, "comb", (0,n))
    self.__checkType(basis, (list, tuple), "comb", length=n)
    list(self.__checkType(_, int, "comb", (1, n)) for _ in basis)
    return self.__comb(n, r, basis)
  def __comb(self, n, r, basis):
    if r > n //2:
      rng = range(1, len(basis) +1)
      for elem in reversed(list(self.__perm(n, [1] *(n -r), 1))):
        yield tuple(basis[idx -1] for idx in rng if idx not in elem) 
    else:
      for elem in self.__perm(n, [1] *r, 1):
        yield tuple(basis[idx -1]  for idx in elem)

  def __perm(self, n, arr, offset):
    """Return permutation or combination arrays. The arr needs to be set to a
       list of one. Set offset to 0 to get permutations or 1 for combinations
       instead."""
    if offset < 0 or n < 0 or len(arr) > n:
      raise Exception("Invalid parameter for perm or comb: %s" %n)
    if len(arr) == 0:
      yield []
    else:
      for recuse in range(offset if offset else 1, n +1):
        if len(arr) > 1:
          for more in self.__perm(n, arr[1:], recuse if offset else 0):
            arr = [recuse] +more
            dup = False
            for idx,elem in enumerate(arr):
              if elem in arr[idx+1:]:
                dup = True
                break
            if not dup:
              yield arr
        else:
          yield [recuse]
  def __checkType(self, arg, typ, method, size=[], length=None):
    """Raise exception if arg not the correct type in method for size."""
    if not isinstance(arg, typ):
      tmp = str(arg)
      if len(tmp) > 9:
        tmp = str(type(arg))
      raise Exception("Invalid parameter type (%s) for %s" %(tmp, method))
    if length:
      size = length
      arg = len(arg)
    if size:
      if isinstance(size, int):
          if size != arg:
            raise Exception("Invalid %s !=%d for %s" %(arg, size, method))
      elif isinstance(size, (list, tuple)):
        if len(size) == 2 and (arg < size[0] or \
              (size[1] > 0 and arg > size[1])):
          raise Exception("Invalid %s !in [%d,%s] for %s" %(arg, size[0],
                          size[1] if size[1] else "..", method))
      elif size:
        raise Exception("Invalid check %s parameters in %s" %(arg, method))

  def seedRand(self):
    """The system does this by default."""
    random.seed(time.time())

  def randRot(self, stage):
    """Generate random rotation pairs, stage=0: single for init, =1: odd
       dmensions only, =2: even dimensions, try for all, drop overlaps."""
    self.__checkType(stage, int, "randRot", (0,2))
    if stage == 0:   # init case
      size = 2
      start = 1
      step = 1
    else:
      size = self.__msgLen *3 //2 +1
      start = 0 if stage == 1 else 2
      step = 2
    rots = []
    for pos in range(1, size):
      pair = [random.randrange(start, self.__dim, step),
              random.randrange(start, self.__dim, step)]
      if stage == 1: # Sender gets odd dimensions
        pair = [pair[0] +1, pair[1] +1]
      if pair[0] > pair[1]:
        tmp = pair[0]; pair[0] = pair[1]; pair[1] = tmp
      if pair not in rots and pair[0] != pair[1]:
        rots.append(pair)
    return rots

  def pack(self, msg):
    """Change ASCII msg into 7 bits -> sign +1/-1 ints."""
    self.__checkType(msg, str, "pack", length=(1,0))
    if len(msg) *7 > self.__msgLen:
      raise Exception("Message it too long by %d character(s)" \
                              %(len(msg) -self.__msgLen //7))
    signs = []
    for ch in msg:
      for bit in format(ord(ch), "07b"):
        signs.append(-1 if bit == "0" else 1)
    if len(signs) < self.__msgLen -7:
      for bit in format(0x25, "07b"): # Add control-d
        signs.append(-1 if bit == "0" else 1)
    for idx in range(len(signs), self.__msgLen):  # pad with random bits
      signs.append(1 if random.randrange(2) else -1)
    return signs

  def unpack(self, signs):
    """Change sign +1/-1 ints -> 7 bits for ASCII msg."""
    self.__checkType(signs, (list, tuple), "unpack", length=self.__msgLen)
    msg = ""
    for idx in range(0, len(signs) -6, 7):
      bits = ""
      for bit in signs[idx :idx +7]:
        bits += "1" if bit == 1 else "0"
      ch = chr(int(bits, 2))
      if ch == chr(0x25): # EOM control-d
        break
      msg += ch
    return msg

  def save(self, struct, filename):
    """Save a python list with some type formatting of the structure."""
    nl = ""
    try:
      with open(filename, 'w') as fp:
        if isinstance(struct, (list, tuple, set)):
          fp.write("[\\\n")
          for val in struct:
            if isinstance(val, (list, tuple, set)):
              fp.write("%s  %s,\n" %(nl, val))
              nl = ""
            elif isinstance(val, str):
              fp.write("  \"%s\",\n" %val)
              nl = ""
            else:
              fp.write("  %s," %val)
              nl = "\n"
          fp.write("]")
        elif isinstance(struct, str):
          fp.write("%s" %struct)
        else:
          fp.write("%s" %struct)
        fp.write("\n")
    except BaseException as e:
      sys.stderr.write('%s: %s\n' %(type(e).__name__, e))
  def _generateSpace(self, store, parts, selfGetStrFn, selfSetStrFn,
                     status=None, pos=0, timeout=0):
    """Test code. Return new rotated 3- or 7-form calibrations as strings."""
    allRots = list((x,y) for x in range(1, self.__dim) \
                         for y in range(x +1, self.__dim +1))
    out = set()
    for offs,calStr in enumerate(store[parts[0] :parts[1]]):
      if status:
        status[pos] = offs
      cal = selfSetStrFn(calStr)
      for rot in allRots:
        self.__setCode(self.__sgns, cal)
        self.__rotate((rot,))
        newCal = selfGetStrFn()
        if newCal not in store:
          out.add(newCal)
      if timeout and timeout < time.time():
        break
    return out

  @staticmethod
  def read(filename):
    """Read a Python structure (generally saved by save())."""
    try:
      with open(filename) as fp:
        return eval(fp.read())
    except Exception:
      typ,var,tb = sys.exc_info()
      raise Exception(typ(str(var).replace("\\\\\\\\", "\\")))

_CalCrypt = CalCrypt.__doc__  # Doco for calculators
################################################################################
if __name__ == '__main__':
  import traceback
  class HelpException(Exception):
    """Don't report this exception."""
    pass
  try:
    options = ("-v", "/v", "-vh", "-hv", "-h", "/?")
    optIdx = (1,5) if sys.platform == "win32" else (0,4)   # (v,h)
    verbose = False
    offset = stage = 0
    while len(sys.argv) > 1 +offset and (sys.argv[1 +offset] in options):
      verbose = sys.argv[1 +offset] in options[:4]
      if sys.argv[1 +offset] in options[-4:]:
        help(CalCrypt)
        raise HelpException
      offset += 1
    if len(sys.argv) > 2 +offset:
      stage = 1
      lev = int(sys.argv[1 +offset])
      base = sys.argv[2 +offset]
      msg = " ".join(sys.argv[3 +offset:])
      if not msg:
        raise Exception("Message is blank")
    elif len(sys.argv) == 2 +offset:
      filename = sys.argv[1 +offset]
      stage, lev, base, signs, cal = CalCrypt.read(filename)
    else:
      raise Exception("Invalid command line arguments")
    crypt = CalCrypt(lev)
    if stage == 1:
      crypt.seedRand()
      crypt.setCode(crypt.pack(msg))
      rot1 = crypt.randRot(2)
      crypt.rotate(rot1)
      crypt.save([rot1], base +"1.key")
      crypt.save([2, lev, base, *crypt.getCode()], base +"1.cry")
      crypt.save([rot1], base +"1.key")
      sys.stdout.write("Transmit file: %s1.cry (Key file: %s1.key)\n" %(base, base))
    elif stage == 2:
      crypt.seedRand()
      crypt.setCode(signs, cal)
      rot2 = crypt.randRot(1)
      crypt.rotate(rot2)
      crypt.save([rot2], base +"2.key")
      crypt.save([3, lev, base, *crypt.getCode()], base +"2.cry")
      sys.stdout.write("Transmit file: %s2.cry (Key file: %s2.key)\n" %(base, base))
    elif stage == 3:
      crypt.setCode(signs, cal)
      rot1 = CalCrypt.read(base +"1.key")[0]
      crypt.rotate(crypt.reverse(rot1))
      crypt.save([4, lev, base, *crypt.getCode()], base +"3.cry")
      sys.stdout.write("Transmit file: %s3.cry\n" %base)
    elif stage == 4:
      crypt.setCode(signs, cal)
      rot2 = CalCrypt.read(base +"2.key")[0]
      crypt.rotate(crypt.reverse(rot2))
      msg = crypt.unpack(crypt.getCode()[0])
      crypt.save(msg, base +".txt")
      sys.stdout.write("Message file: %s.txt\n" %base)
    else:
      raise Exception("Corrupt transmitted file")
  except Exception as e:
    cmd = ("py " if sys.platform == "win32" else "python[3] ") +"CalCrypt.py"
    print("Start usage: %s Level Basename \"Message\" (or unquoted)" %cmd)
    print("Other usage: %s \"Filename\"" %cmd)
    print("Help option:  %s %s [%s]" %(cmd, options[optIdx[1]], options[optIdx[0]]))
    print("Desc: Start Level = 3..10 determines the maximum message length.")
    desc = ( \
      "Each usage creates a file to transmit to the other side and Other use",
      "processes the file with the last transmit being the decoded message.",
      "Messaage lengths are: 1, 5, 22, 93, 381, 1542, 6205, 24893 chars.")
    for txt in desc:
      print("      " +txt)
    if type(e) != HelpException:
      if verbose:
        traceback.print_exc()
      else:
        sys.stderr.write('%s: %s\n' %(type(e).__name__, e))
\end{verbatim}}

\newpage
\section*{Appendix B - Test Code}
{\fontsize{10pt}{10pt}\selectfont
}

\newpage
\section*{Appendix C - 103 Sedenion Quasi-Calibrations}
\tac{!ht}{qs1}{Sedenion Quasi-Calibrations}
{\fontsize{10pt}{10pt}\selectfont
 \begin{tabular}{ll}
  $\PB{122}$ & $\e{123} +\e{149} +\e{15E} +\e{16D} -\e{17B} +\e{18A} +\e{1CF} +\e{24B} +\e{25C} +\e{268} +\e{279} +\e{2AD}$\\
    &$+\e{2EF} +\e{347} +\e{35F} +\e{36A} -\e{38D} +\e{39B} +\e{3CE} +\e{456} +\e{48C} +\e{4AF} +\e{4DE} +\e{57A} $\\
    &$+\e{58B} +\e{59D} +\e{67F} +\e{69E} +\e{6BC} +\e{78E} +\e{7CD} +\e{89F} +\e{9AC} +\e{ABE} +\e{BDF}$ \\
  $\PB{124}$ & $\e{12C} +\e{138} -\e{14F} +\e{15E} +\e{16D} +\e{17A} +\e{19B} +\e{23B} +\e{24E} +\e{25F} +\e{267} +\e{289}$ \\
    &$+\e{2AD} +\e{34A} +\e{356} +\e{37F} +\e{39C} +\e{3DE} +\e{45C} +\e{469} +\e{478} +\e{4BD} +\e{57B} +\e{58D} $ \\
    &$+\e{59A} +\e{68E} -\e{6AC} +\e{6BF} +\e{79E} +\e{7CD} +\e{8AF} +\e{8BC} +\e{9DF} +\e{ABE} +\e{CEF}$ \\
  $\PB{126}$ & $\e{12B} +\e{135} +\e{14D} +\e{167} +\e{189} +\e{1AC} +\e{1EF} +\e{239} +\e{246} +\e{258} +\e{27D} +\e{2AF}$ \\
    &$+\e{2CE} +\e{34F} +\e{36A} -\e{37C} +\e{38B} +\e{3DE} +\e{45E} +\e{47B} +\e{48C} +\e{49A} +\e{56C} +\e{57A}$ \\
\end{tabular}}\tae
\tac{!ht}{qs2}{Sedenion Quasi-Calibrations (continued)}
{\fontsize{10pt}{10pt}\selectfont
 \begin{tabular}{ll}
  $\PB{166}$ &$\e{128} +\e{139} +\e{14E} +\e{15F} +\e{16C} +\e{17D} +\e{1AB} +\e{23A} +\e{246} +\e{257} +\e{29B} +\e{2CE}$\\
    &$ +\e{2DF} +\e{347} +\e{356} +\e{38B} +\e{3CF} +\e{3DE} +\e{45A} +\e{48C} +\e{49D} -\e{4BF} +\e{58D} +\e{59C}$\\
    &$ +\e{5BE} +\e{67A} +\e{68E} +\e{69F} +\e{6BD} +\e{78F} +\e{79E} +\e{7BC} +\e{89A} +\e{ACD} +\e{AEF}$\\
  $\PB{168}$ &$\e{123} +\e{14B} +\e{15A} +\e{167} +\e{189} +\e{1CD} +\e{1EF} +\e{246} +\e{259} +\e{27B} +\e{28A} +\e{2CE}$\\
    &$ +\e{2DF} +\e{347} +\e{358} +\e{36B} +\e{39A} +\e{3CF} +\e{3DE} +\e{45F} +\e{48C} +\e{49D} -\e{4AE} +\e{56D} $\\
    &$+\e{57C} +\e{5BE} +\e{68E} +\e{69F} +\e{6AC} +\e{78F} +\e{79E} +\e{7AD} +\e{8BD} +\e{9BC} +\e{ABF}$\\
  $\PB{170}$ &$\e{123} +\e{14A} +\e{15B} +\e{167} +\e{189} +\e{1CD} +\e{1EF} +\e{246} +\e{258} +\e{27A} +\e{29B} +\e{2CE}$\\
    &$ -\e{2DF} +\e{347} +\e{359} +\e{36A} +\e{38B} +\e{3CF} +\e{3DE} +\e{45E} +\e{48C} +\e{49D} +\e{4BF} +\e{56C}$\\
\end{tabular}}\tae
\tac{!ht}{qs3}{Sedenion Quasi-Calibrations (continued)}
{\fontsize{10pt}{10pt}\selectfont
 \begin{tabular}{ll}
  $\PB{214}$ &$\e{123} +\e{145} +\e{167} +\e{189} +\e{1AD} +\e{1BC} +\e{1EF} +\e{246} +\e{257} +\e{28A} +\e{29D} +\e{2BF}$\\
    &$ +\e{2CE} +\e{347} +\e{356} +\e{38D} +\e{39A} -\e{3BE} +\e{3CF} +\e{48C} +\e{49B} +\e{4AE} +\e{4DF} +\e{58B}$\\
    &$ +\e{59C} +\e{5AF} +\e{5DE} +\e{68E} +\e{69F} +\e{6AC} +\e{6BD} +\e{78F} +\e{79E} +\e{7AB} +\e{7CD}$\\
  $\PB{216}$ &$\e{123} +\e{145} +\e{16A} -\e{17B} +\e{189} +\e{1CD} +\e{1EF} +\e{246} +\e{25A} +\e{278} +\e{29B} +\e{2CE}$\\
    &$ +\e{2DF} +\e{34A} +\e{356} +\e{379} +\e{38B} +\e{3CF} +\e{3DE} +\e{47E} +\e{48C} +\e{49D} +\e{4BF} +\e{57F}$\\
    &$ +\e{58D} +\e{59C} +\e{5BE} +\e{67C} +\e{68E} +\e{69F} +\e{6BD} +\e{7AD} +\e{8AF} +\e{9AE} +\e{ABC}$\\
  $\PB{218}$ &$\e{123} +\e{145} +\e{168} +\e{179} +\e{1AB} +\e{1CD} +\e{1EF} -\e{246} +\e{258} +\e{27A} +\e{29B} +\e{2CE}$\\
    &$ +\e{2DF} +\e{348} +\e{356} +\e{37B} +\e{39A} +\e{3CF} +\e{3DE} +\e{47C} +\e{49D} +\e{4AE} +\e{4BF} +\e{57D}$\\
\end{tabular}}\tae
\tac{!ht}{qs4}{Sedenion Quasi-Calibrations (continued)}
{\fontsize{10pt}{10pt}\selectfont
 \begin{tabular}{ll}
  $\PB{262}$ &$\e{123} +\e{14D} +\e{15C} +\e{167} +\e{189} +\e{1AB} +\e{1EF} +\e{246} +\e{25F} +\e{27D} +\e{28A} +\e{29B}$\\
    &$ +\e{2CE} +\e{347} +\e{35E} +\e{36D} +\e{38B} +\e{39A} +\e{3CF} +\e{459} +\e{48C} +\e{4AE} +\e{4BF} +\e{56B}$\\
    &$ +\e{57A} +\e{58D} +\e{68E} +\e{69F} +\e{6AC} +\e{78F} +\e{79E} +\e{7BC} +\e{9CD} +\e{ADF} +\e{BDE}$\\
  $\PB{264}$ &$\e{123} +\e{145} +\e{16D} +\e{17C} +\e{189} +\e{1AB} +\e{1EF} +\e{24C} +\e{257} +\e{26E} +\e{28A} +\e{29B}$\\
    &$ +\e{2DF} +\e{347} -\e{35C} +\e{36F} +\e{38B} +\e{39A} +\e{3DE} +\e{468} +\e{49D} +\e{4AE} +\e{4BF} +\e{569} $\\
    &$+\e{58D} +\e{5AF} +\e{5BE} +\e{67B} +\e{6AC} +\e{78F} +\e{79E} +\e{7AD} +\e{8CE} +\e{9CF} +\e{BCD}$\\
  $\PB{266}$ &$\e{123} +\e{14D} +\e{15C} -\e{167} +\e{189} +\e{1AB} +\e{1EF} +\e{246} +\e{25F} +\e{27D} +\e{28A} +\e{29B}$\\
    &$ +\e{2CE} +\e{347} +\e{35E} +\e{36D} +\e{38B} +\e{39A} +\e{3CF} +\e{459} +\e{48C} +\e{4AE} +\e{4BF} +\e{56B}$\\
\end{tabular}}\tae
\tac{!ht}{qs5}{Sedenion Quasi-Calibrations (continued)}
{\fontsize{10pt}{10pt}\selectfont
 \begin{tabular}{ll}
  $\PB{310}$ &$\e{123} +\e{145} +\e{16D} +\e{17C} +\e{189} +\e{1AF} +\e{1BE} +\e{247} +\e{25C} +\e{26E} +\e{28A} +\e{29F}$\\
    &$ +\e{2BD} +\e{34C} +\e{357} +\e{36B} +\e{38F} +\e{39A} +\e{3DE} +\e{468} +\e{49D} +\e{4AE} -\e{4BF} +\e{569}$\\
    &$ +\e{58D} +\e{5AB} +\e{5EF} +\e{67A} +\e{6CF} +\e{78E} +\e{79B} +\e{7DF} +\e{8BC} +\e{9CE} +\e{ACD}$\\
  $\PB{312}$ &$\e{12E} +\e{13D} +\e{148} +\e{159} +\e{16C} +\e{17B} +\e{1AF} +\e{23C} +\e{24F} +\e{257} +\e{26D} +\e{28A}$\\
    &$ +\e{29B} +\e{347} +\e{35F} +\e{36E} +\e{38B} +\e{39A} +\e{45C} +\e{469} +\e{4AE} +\e{4BD} +\e{568} +\e{5AD} $\\
    &$+\e{5BE} +\e{67A} -\e{6BF} +\e{78D} +\e{79E} +\e{7CF} +\e{89C} +\e{8EF} +\e{9DF} +\e{ABC} +\e{CDE}$\\
  $\PB{314}$ &$\e{125} +\e{13C} +\e{14D} +\e{169} +\e{178} +\e{1AF} +\e{1BE} +\e{237} +\e{246} +\e{28C} +\e{29D} -\e{2AE}$\\
    &$ +\e{2BF} +\e{34B} +\e{358} +\e{36F} +\e{39A} +\e{3DE} +\e{459} +\e{47F} +\e{48A} +\e{4CE} +\e{56D} +\e{57C}$\\
\end{tabular}}\vskip 120ex\tae
\end{document}